\documentclass[a4paper,11pt,oneside]{amsart}

\usepackage{amssymb}  
\usepackage{latexsym} 
\usepackage{comment}
\usepackage{color}
\usepackage{colonequals}
\usepackage{amsmath}

\usepackage[left=30mm, right=30mm, top=3cm, bottom=3cm]{geometry}

\usepackage{setspace}
\definecolor{darkgreen}{rgb}{0,0.5,0}
\usepackage[colorlinks, citecolor=darkgreen, linktocpage ]{hyperref}

\usepackage[alphabetic,lite]{amsrefs} 

\usepackage[all]{xy} 
\usepackage{enumerate} 

\newtheorem{theorem}{Theorem}
\newtheorem{proposition}[theorem]{Proposition}
\newtheorem*{question}{Question}
\newtheorem{corollary}[theorem]{Corollary}
\newtheorem{lemma}[theorem]{Lemma}

\theoremstyle{definition}
\newtheorem{definition}[theorem]{Definition}

\theoremstyle{remark}
\newtheorem{remark}[theorem]{Remark}

\DeclareMathOperator{\Fix}{Fix}

\DeclareMathOperator{\Aut}{Aut}

\DeclareMathOperator{\GL}{GL}

\numberwithin{equation}{section}

\title{Rigid Manifolds of general type with non-contractible universal cover}

\makeatletter

\@addtoreset{equation}{section}
\makeatother

\author{Davide Frapporti, Christian Gleissner}

\address{Davide Frapporti,
Universit\`a  degli Studi di Genova, Dipartimento di Matematica - DIMA;\newline
Via Dodecaneso 35, I-16146 Genova, Italy.}
\email{davide.frapporti@edu.unige.it}
\address{ Christian Gleissner ,
 University of Bayreuth, Lehrstuhl Mathematik VIII; \newline
Universit\"atsstra\ss e 30, D-95447 Bayreuth, Germany.	
}
\email{Christian.Gleissner@uni-bayreuth.de}

\date{\today}


\thanks{The authors thank I.~Bauer, F.~Catanese, S.~Coughlan and F.~Fallucca for their interest and fruitful conversations.\\
The first author is member of  G.N.S.A.G.A. of I.N.d.A.M. and acknowledges support of 
 the ERC Advanced grant n. 340258-TADMICAMT}

\date{\today}

\begin{document}

\begin{abstract} 
For each $n\geq 3$ we give  examples of  infinitesimally  rigid projective manifolds of general type of dimension $n$ with  non-contractible universal cover. We provide examples with projective  and examples with  non-projective universal cover.
\end{abstract}

\maketitle

\section*{Introduction}

In \cite{rigidity} several notions  of \emph{rigidity} have been discussed, the relations among them have been studied and many questions and conjectures have been proposed. 
In particular the authors showed  
that a rigid compact complex surface has Kodaira dimension $- \infty$ or $2$, and observed that all known examples of rigid surfaces of general type  
are $K(\pi,1)$ spaces. Recall that a CW complex  with fundamental group $\pi$ is called  \emph{$K(\pi,1)$ space} if its universal cover is contractible, and that 
these spaces have the property that their homotopy type is uniquely determined by their fundamental group  (cf.~\cite[\S 1.B]{Hatcher}). 
This implies that the topological invariants, such as  homology and cohomology, are determined by $\pi$. 
In \cite{rigidity} the following natural question has been posed. 

\begin{question}
Do there exist infinitesimally rigid  surfaces of general type with non-contractible universal cover?
\end{question}

The aim of this paper is to give a positive answer for the analogous question in higher dimensions. More precisely, we construct for each $n\geq 3$  an  infinitesimally  rigid  manifold of general type  of dimension $n$ with  non-contractible universal cover. 
For surfaces the question remains open.
We  recall now  the notions of rigidity that are relevant for our purposes.
\begin{definition}\label{rigid} \

Let $X$ be a compact complex manifold of dimension $n$. 
\begin{enumerate}
\item A {\em deformation of $X$} is a  proper smooth holomorphic map of pairs $$f \colon (\mathfrak{X},X)  \rightarrow (\mathcal{B}, b_0),
$$ 
where $(\mathcal B,b_0)$ is a connected (possibly not reduced) germ of a complex space.

\item $X$ is said to be  {\em  rigid}  if for each deformation of $X$,
$
f \colon (\mathfrak{X},X)  \rightarrow (\mathcal B, b_0)
$ 
there is an open neighbourhood $U \subset \mathcal B$ of $b_0$ such that $X_t := f^{-1}(t) \simeq X$ for all $t \in U$.
\item  $X$ is said to be  {\em infinitesimally rigid} if 
$H^1(X, \Theta_X) = 0$,
where $\Theta_X$ is the sheaf of holomorphic vector fields on $X$.
\item 
$X$ is said to be (infinitesimally) \'etale rigid if all finite \'etale covers $f\colon Y \to X$ are (infinitesimally) rigid. 
\end{enumerate}
\end{definition}

\begin{remark}\label{kuranishi} 

i) By Kodaira-Spencer-Kuranishi theory every infinitesimally rigid manifold is rigid.
The converse does not hold in general as it was shown in \cite{BP18} and \cite{BBP20}  (cf.~also \cite{kodairamorrow}).

ii) Beauville surfaces are examples of rigid, but not \'etale rigid manifolds (see \cite{Cat00}).
\end{remark}

Both the examples constructed in \cite{BP18} and Beauville surfaces are product quotient varieties, i.e.~(resolutions of singularities of) finite  quotients of product of curves  with respect to a  holomorphic group action.
In recent years, product quotients turned out to be a very fruitful source of examples  of rigid complex manifolds with additional properties. 
Besides the examples above,  we mention  \cite{BG20},  where the authors construct the first examples of  rigid complex manifolds 
with Kodaira dimension 1 in arbitrary dimension $n\geq 3$, and \cite{BG21} where they constructed new rigid three- and four-folds  with Kodaira dimension 0. 
We refer to \cite{CF18,FGP18,FG19,GPR18,LP16,LP18} for other interesting examples of product quotient varieties.

The manifolds we construct are also product quotients.  More precisely,  inspired by the construction in \cite{BP18} in Section \ref{sec:example}
 we consider  for each $n \geq 3$ and $d\geq 4$, even and not divisible by 3 the $n$-fold product  $C^n$ of the  Fermat curve $C$ of degree $d$ together with a suitable action  of $\mathbb Z_d^2$.
The quotient $X_{n,d} := C^n/\mathbb Z_d^2$ is 
 a normal projective variety  with isolated cyclic quotient singularities of type $\frac{1}{2}(1, \ldots,1)$, Kodaira dimension $n$ and 
\[
H^1(X_{n,d}, \Theta_{X_{n,d}}) = H^1(C^n, \Theta_{C^n})^{\mathbb Z_d^2} = 0.
\]
Blowing up the singular points, we obtain a resolution $\widehat{X_{n,d}} \rightarrow {X_{n,d}}$ such that 
$H^1({X_{n,d}}, \Theta_{X_{n,d}}) = H^1(\widehat{X_{n,d}}, \Theta_{\widehat{{X_{n,d}}}})$. Therefore,  
${\widehat{{X_{n,d}}}}$ is an infinitesimally rigid projective manifold of general type.  

In Section \ref{sec:UC} we show that the 
universal cover $U_{n,d}$ of $\widehat{X_{n,d}}$ is non-contractible since it contains several $\mathbb P^{n-1}$ (see Propostion \ref{Prop:non-contr}). We then discuss the finiteness of the fundamental group 
$\pi_1({X_{n,d}}) = \pi_1(\widehat{{X_{n,d}}})$.
The crucial ingredient  here is 
Armstrong's description of the fundamental group of a quotient space \cite{armstrong} adapted to product quotients by  \cite{BCGP12}.
The finiteness of $\pi_1(\widehat{X}_n)$  is equivalent to  the finiteness of certain groups (Proposition \ref{prop:fin_crit}: Finiteness criterion). This allows us to prove the following.
\begin{theorem} For each $n\geq 3$,  $d\geq 4$, even and not divisible by 3   there exists an infinitesimally rigid projective $n$-dimensional manifold of general type $\widehat{X_{n,d}}$, whose universal cover $U_{n,d}$ is non-contractible.
Moreover, the universal cover $U_{n,d}$ is projective if and only if $d=4$.
\end{theorem}

The construction actually works also for $n=2$: the surface $\widehat{X_{2,4}}$ is not rigid, whereas the surface
 $\widehat{X_{2,d}}$  for $d\geq 8$ is rigid but not infinitesimally rigid (see \cite{BP18}), and its universal cover is non-contractible. 

\
 
\textbf{Notation.}
We work over the field of complex numbers, and we denote by 
$\mathbb Z_n$ the cyclic group of order $n$ and by $\zeta_n$ a primitive $n$-th root of unity. 
The rest of the notation is standard in complex algebraic geometry.

\section{The families}\label{sec:example}

Let $C_d:=\lbrace x_0^d+x_1^d+x_2^d=0\rbrace \subset \mathbb P^2$  be the  Fermat curve of degree $d$. Consider the group action 
\[
\phi_1\colon \mathbb Z_d^2 \to \Aut(C),  \qquad (a,b) \mapsto [(x_0:x_1:x_2)  \mapsto (\zeta_d^a x_0: \zeta_d^b x_1 : x_2)]\,.
\]
There are  $3d$  points on $C_d$  with non-trivial stabilizer. They form three orbits of length $d$. 
A representative  of each orbit and a generator of the corresponding stabilizer is given in the table below:

\begin{center}
\setlength{\tabcolsep}{1pt}

\begin{tabular}{c | c | c | c }

 point  & $ ~(0:1:\zeta_{2d})  ~ $  & $~ (1:0:\zeta_{2d}) ~ $ ~   & $ ~(1: \zeta_{2d} : 0) ~ $   \\
\hline
$~~$ generator $~~$  & $ (1,0) $  &  $  (0,1)  $ &  $ (1,1) $ 
\end{tabular}
\end{center}
Hence the quotient map 
\[f \colon C_d \to \mathbb P^1, \qquad (x_0:x_1:x_2) \mapsto (x_0^d:x_1^d)\]
 is branched in $(0:1)$, $(1:0)$ and $(1:-1)$, each  with branch index $d$.

\subsection{The singular quotients $X_{n,d}$} 

From now on we fix $d\geq 4$, even and not divisible by 3, and denote $C_d$ simply by $C$.
Let $A$ be the automorphism of $\mathbb Z_d^2$  given by the matrix
\[
\begin{pmatrix} 1 & -2 \\ 2 & -1 \end{pmatrix} \in \GL(2,\mathbb Z_d)\,,
\]
 and let $\phi_2 := \phi_1 \circ A^{-1}$. 
 For each $n \geq 2$ consider the $\mathbb Z_d^2$  diagonal action  on  $C^n$ defined by 
\[
g( z_1,\ldots, z_n ):=\big(\phi_1(g) \cdot   z_1 , \phi_2(g) \cdot  z_2 , \phi_2(g) \cdot  z_3, \ldots, \phi_2(g) \cdot  z_n \big)
\] and let  $X_{n,d}$ be the quotient variety $X_{n,d}:=C^n/\mathbb Z_d^2$. 

\begin{remark}\label{rem:defH}
The diagonal action is not free, indeed 
\[\Fix(\phi_1(g))\neq \emptyset  \text{ and } \Fix(\phi_2(g)) \neq \emptyset 
\Longleftrightarrow
 g\in H:=\left\langle \left(\frac d2,0\right), \left(0,\frac d2\right)\right\rangle.\]

Noting  that
${\phi_1}_{|H}=  {\phi_2}_{|H} $, we see that a point  $( z_1,\ldots, z_n )\in C^n$ has a non-trivial stabilizer if and only if  all its coordinates $z_i$ belong to one and only one of the three $\mathbb Z_d^2$-orbits displayed in the table above.

\end{remark}

\begin{proposition}\label{propOfXn}
For $n\geq 3$ the projective  variety $X_{n,d}$ is infinitesimally rigid and  of  general type. The singular locus of $X_{n,d}$ consists of 
$6\cdot d^{n-2}$  cyclic quotient singularities of type $\frac{1}{2}(1,\ldots,1)$. 
\end{proposition}

\begin{proof}

By Remark \ref{rem:defH} there are  $3\cdot d^n$ points on $C^n$ with non-trivial stabilizer,
each generated by one of the order 2 elements in $\mathbb Z_d^2$.
 Thus,   $X_{n,d}$ has 
 $(3\cdot d^n)/(d^2/2)= 6\cdot d^{n-2}$ singularities of type  $\frac{1}{2}(1,\ldots,1)$.

   These singularities  are terminal if  
 $n\geq 3$, see  \cite[p.~376 Theorem]{R87}. Since the  
quotient map 
 $C^n \to X_{n,d}$ is quasi-\'etale,  $g(C)=(d-1)(d-2)/2\geq 3$ and  $X_{n,d}$ is terminal, its Kodaira dimension is 
 $\kappa(X_{n,d})=\kappa(C^n)=n$ (cf.~\cite[p.~51]{Cat07}).
 
According to  Schlessinger \cite{schlessinger},  isolated quotient singularities in dimension at least three are rigid,  i.e.
 $\mathcal Ext^1(\Omega_{X_{n,d}}^1, \mathcal O_{X_{n,d}}) =0$. Thus 
 the local-to-global ${\rm Ext}$ spectral sequence yields 
\[H^1(X_{n,d}, \Theta_{X_{n,d}}) \simeq {\rm Ext}^1(\Omega_{X_{n,d}}^1, \mathcal O_{X_{n,d}})\,.\]
Hence it suffices to verify  that $X_{n,d}$ has no equisingular deformations. 
Since $g(C)\geq 3$ we have $H^0(C, \Theta_C)=0$,  hence by K\"unneth formula we get  
\[H^1(C^n, \Theta_{C^n})=
\bigoplus_{i=1}^n H^1(C, \Theta_{C})\,.\]
Using the fact that  the quotient map  $ C^n \to X_{n,d}$ is quasi-\'etale and the action is diagonal, we obtain
\[ H^1(X_{n,d}, \Theta_{X_{n,d}})= H^1(C^n, \Theta_{C^n})^{\mathbb Z_d^2}= 
\bigoplus_{i=1}^n H^1(C, \Theta_{C})^{\mathbb Z_d^2}\,.\]
The branch locus $B$ of $f\colon C \to C/\mathbb Z_d^2 \simeq \mathbb P^1$ consists of 3
points $p_i$ with branch indices $m_{p_i}=d$, thus by \cite[Ex.~VI.12]{Beau78} we have 
\begin{eqnarray*}\dim H^1(C, \Theta_{C})^{\mathbb Z_d^2}&=&\dim H^0(C, 2K_{C})^{\mathbb Z_d^2}=
h^0(\mathbb P^1, 2K_{\mathbb P^1} + \sum_{p_i\in B}  p_i \cdot \lfloor2 (1-\frac 1{m_{p_i}})\rfloor)\\&=&
h^0(\mathbb P^1, \mathcal O(-1))=0 \,.\end{eqnarray*}
\end{proof}

\subsection{Resolution of singularities of type $\frac12(1,\ldots, 1)$} 

\begin{proposition}\label{Prop_Res}
A singularity  $U:=\mathbb C^n/\mathbb Z_2$ of type $\frac{1}{2}(1,\ldots,1)$ admits a  resolution 
$\rho \colon \widehat{U} \to U$ by a single  blow-up, with exceptional prime divisor $\mathbb P^{n-1}$.
If  $n\geq 3$,
\[
\rho_{\ast} \Theta_{\widehat{U}}=  \Theta_{U} \quad \makebox{and} \quad R^1\rho_{\ast} \Theta_{\widehat{U}}=0. 
\]
\end{proposition}

For a proof we refer to \cite[proof of Theorem 4]{schlessinger}, see also \cite[Corollary 5.9, Proposition 5.10]{BG20}.

\begin{remark}[see {\cite[Remark 5.4]{BG20}}]\label{rem:res}
Both properties are not obvious and in general even false.
For any resolution $\rho: Z'\to Z$ of a normal variety $Z$, the direct image
$\rho_*\Theta_{Z'}$ is a subsheaf of the reflexive sheaf $\Theta_Z$,
and this inclusion is in general strict: e.g.~take the blow-up of the origin of $\mathbb C^2$. 

The vanishing of $R^1\rho_*\Theta_{Z'}$ is also not automatic: take the resolution of an $A_1$ surface singularity (i.e.~$\frac 12 (1,1)$) by a $-2$ curve, then $R^1\rho_*\Theta_{Z'}$ is a skyscraper sheaf at the singular point with value $H^1(\mathbb P^1,\mathcal O(-2))\cong \mathbb C$. More generally, for canonical ADE surface singularities $R^1\rho_*\Theta_{Z'}$ is never zero, cf. \cite{BW74,pinkham,schlessinger}.

\end{remark} 

\begin{corollary}\label{TheReso}
Let $Z_n$ be a projective variety of dimension $n \geq 3$ with only singularities
of type $\frac{1}{2}(1,\ldots,1)$. Then there 
 exists a resolution $\rho \colon \widehat{Z}_n \to Z_n$, such that 
 \[
 H^1(Z_n, \Theta_{Z_n}) \simeq H^1(\widehat{Z}_n, \Theta_{\widehat{Z}_n}).  
 \]
 In particular, if $Z_n$ is infinitesimally rigid, so is $\widehat{Z}_n$. 
 \end{corollary}

\begin{proof}
Since the singularities of $Z_n$ are isolated,  we resolve them simultaneously using Proposition \ref{Prop_Res}  and we get  a resolution $\rho \colon \widehat{Z}_n \to Z_n$ having the same properties:  
\[
\rho_{\ast} \Theta_{\widehat{Z}_n}=  \Theta_{Z_n} \quad \makebox{and} \quad R^1\rho_{\ast} \Theta_{\widehat{Z}_n}=0. 
\]
Leray's spectral sequence  implies  $H^1(\widehat{Z}_n,\Theta_{\widehat{Z}_n}) \simeq H^1(Z_n, \Theta_{Z_n})$.  
\end{proof}

By the corollary, for $n\geq 3$  there exists a resolution  $\widehat{ X_{n,d}}\to  X_{n,d}$ of the singularities of $X_{n,d}$, which is infinitesimally rigid.
By Remark \ref{rem:res},  for $n=2$ the minimal resolution   $\widehat{ X_{2,d}}$ of  $X_{2,d}$ is not infinitesimally rigid, nevertheless the main theorem of \cite{BP18} shows that $\widehat{ X_{2,d}}$  is rigid for $d\geq 8$, whereas 
$\widehat{ X_{2,4}}$  is a numerical Campedelli surface, whose Kuranishi family has dimension $6$.

\subsection{Non-\'etale infinitesimally rigidity}
 We conclude this section constructing  an \'etale cover of $\widehat{X_{n,d}}$ which is not infinitesimally rigid, thus 
$\widehat{X_{n,d}}$ is not \'etale infinitesimally rigid.

Let $H:=\left\langle \left(\frac d2,0\right), \left(0,\frac d2\right)\right\rangle$ be as in Remark \ref{rem:defH}.

\begin{lemma}\label{lem:Y}
Let $Y_{n,d}:=C^n/H$ be the quotient with respect to the restricted diagonal action, then:
\begin{enumerate}
\item
The natural morphism $\psi \colon Y_{n,d} \to X_{n,d}$ 
is an unramified Galois cover with group $\mathbb Z_{d/2}^2$. 
\item
$h^1(Y_{n,d}, \Theta_{Y_{n,d}})= 
3n\cdot \left(\frac{d^2-2d}8\right)$.
\end{enumerate}
\end{lemma}

\begin{proof}
(1)  Since $H$ is a normal subgroup of $\mathbb Z_d^2$  the map $\psi$ is a Galois cover with
group $\mathbb Z_d^2/H\cong \mathbb Z_{d/2}^2$.
By Remark \ref{rem:defH} the stabilizer of a point $z \in C^n$  with respect to the $\mathbb Z_d^2$-action is contained in $H$,
whence the map $\psi$ is unramified.

(2) Since $C\to C/H $ is branched in $\frac{3d}2$ points and 
$g(C/H)=\frac{(d-2)(d-4)}8$, we have
\[
\dim\big(H^1(C^n, \Theta_{C^n})^H)= n \cdot \dim\big(H^1(C, \Theta_{C})^H \big)= 3n\left(\frac{d^2-2d}8\right)\, \]
 arguing as in 
Proposition \ref{propOfXn}.
\end{proof}

\section{The universal cover of $\widehat{ X_{n,d}}$}\label{sec:UC}

In this section we prove  that the universal cover $U_{n,d}$ of $\widehat{ X_{n,d}}$ is non-contractible, and then we discuss whether it is projective or not.

\begin{proposition}\label{Prop:non-contr}
Let $X$ be a  compact K\"ahler manifold, containing a $\mathbb P^m$. 
Then the universal cover $U$ of $X$ is non-contractible.
\end{proposition}
\begin{proof} 
Since $\mathbb P^m$ is simply connected, the inclusion map $i\colon \mathbb P^m \hookrightarrow X$ lifts to a map $f\colon\mathbb P^m \to U$. 
Looking for a contradiction, assume that $U$ is contractible, then $f$ is homotopic to a constant map, therefore  the inclusion 
$i$ is also  homotopic to a constant map.
In particular we see that the induced linear map $i^*: H^2(X, \mathbb C) \to H^2(\mathbb P^m, \mathbb C)$ is the zero map.
Now let $[\omega]$  be a K\"ahler class of  $X$. Its restriction $i^*([\omega])$
 is a K\"ahler class of $\mathbb P^m$, whence non zero, contradiction.
\end{proof}
\begin{corollary} The universal cover $U_{n,d}$ of $\widehat{ X_{n,d}}$ is non-contractible.
\end{corollary}

\begin{proof}
By Proposition \ref{Prop_Res} $\widehat{ X_{n,d}}$ contains several $\mathbb P^{n-1}$.
\end{proof}

\begin{remark} By Lemma \ref{lem:Y} the universal cover $U_{n,d}$ of $\widehat{X_{n,d}}$ is not infinitesimally rigid.
\end{remark}

\subsection{ The  Fundamental Group}\label{sec:fg} \

In this section we  discuss the finiteness of the fundamental group 
$\pi_1(\widehat{ X_{n,d}}) $. 
In order to do this we use the main theorem of \cite{armstrong} in the case of  product quotient varieties following \cite{BCGP12,DP10}. We briefly recall their strategy and we refer to them for further details.

Let $G$ be a finite group acting diagonally on a product  $Z:=C_1\times \ldots \times C_n$ of curves of genus at least 2, and consider the group $\mathbb G$ of all possible lifts of  automorphisms  induced by 
the action of $G$ on $Z$  to the universal cover $u:\mathbb H^n \to Z$.
The group $\mathbb G $ acts properly discontinuously on  $\mathbb H^n$ 
and $u$ is equivariant with respect to the natural map $\mathbb G\to G$, hence  we have 
an isomorphism $\mathbb H^n/\mathbb G \cong Z/G  $.
Since $\mathbb H^n$ is simply connected we can apply Armstrong's results (see \cite{armstrong}) 
and get the following.

\begin{proposition}\label{prop:fund}
Let $\Fix(\mathbb{G})$ be the normal subgroup of $\mathbb G$ generated by the elements having non-empty fixed locus. Then 
\[\pi_1(Z/G)=\mathbb G / \Fix(\mathbb G)\,.\]

\end{proposition}

Assume that the $G$-action on $Z$ restricts to a faithful action $\phi_i$ on each factor $C_i$.
Let $\mathbb T_i $ be the group 
 of all possible lifts of  automorphisms  induced by 
the action of $G$ on $C_i$  to the universal cover $\mathbb H$ of $C_i$, and  let $\varphi_i :\mathbb T_i\to G$ be  the natural map.
In this setting, the above group $\mathbb G$ is the preimage of the diagonal subgroup $\Delta_G \subset G^n$ under
$\varphi_1\times \ldots\times \varphi_n$:
\[\mathbb G =\{(x_1,\ldots,x_n)\in \mathbb T_1\times \cdots \times \mathbb T_n\mid 
\varphi_1(x_1)=\ldots=\varphi_n(x_n)\}\,.\]

There is also  a similar description of $\mathbb G$  in the non-faithful case, see  \cite[Proposition 3.3]{DP10}.

\begin{remark}\label{rem:OSG}
i) The group $\mathbb T_i$ has a simple presentation (see also \cite[Example 29]{catMod}):
let $g'$ be the genus of $C_i/G$ and $m_1, \ldots, m_r$ be the ramification indices of the 
branch points of the covering map $C_i \to C_i/G$, then 
\[\mathbb T_i= \mathbb{T}(g';m_1,\ldots ,m_r):=
\langle a_1,b_1,\ldots, a_{g'},b_{g'},c_1, \ldots, c_r \mid
 c_1^{m_1}, \ldots, c_r^{m_r},\prod_{i=1}^{g'} [a_i,b_i]\cdot c_1 \cdots c_r\rangle\,.\]
ii) The  group $ \mathbb{T}(g';m_1,\ldots ,m_r)$ is called the
\textit{orbifold surface group} of type $[g';m_1, \ldots, m_r]$.

The non-trivial stabilizers of the $\mathbb T_i$-action on $\mathbb H$ are cyclic and  generated by the  conjugates of the elements $c_k$.
The restriction of $\varphi_i$ to each one of these subgroups is an isomorphism onto its image, which is the stabilizer of a point in $C_i$. Conversely, all non-trivial stabilizers of the $G$-action on $C_i$  are of this form
(see \cite{BCGP12}).

\end{remark}

 \begin{definition}
 Let $L_i\subset \mathbb T_i$ be set of  the elements 
$ c_j^{l_j} \in  \mathbb T_i$ such that $\varphi_i(c_j^{l_j}) \in G$ has non-empty fixed locus on $Z=C_1\times \ldots \times C_n$, where  $j \in \{1,\ldots, r\}$ and $l_j \in \{1,\ldots, m_j-1\}$.

We denote by $\langle\langle L_i\rangle\rangle_{\mathbb T_i}$  the normal subgroup of $\mathbb T_i$ generated by $L_i$.
\end{definition}

\begin{proposition}[Finiteness criterion]\label{prop:fin_crit}
The group $\pi_1(Z/G)=\mathbb G / \Fix(\mathbb G)$ is finite if and only if
the groups $\mathbb T_i/\langle\langle L_i\rangle\rangle_{\mathbb T_i}$ are finite.
\end{proposition}

\proof
According to \cite[pag.1018-1019]{BCGP12} the group $\mathbb G / \Fix(\mathbb G)$ fits in an exact sequence
\[1\to E \to \mathbb G / \Fix(\mathbb G)\to \mathbf H\to 1,\]
where $E$ is a finite group and $\mathbf{H}$ is a subgroup of finite index of  the product
\[\mathbb T_1/\langle\langle L_1\rangle\rangle_{\mathbb T_1} \times \cdots \times \mathbb T_n/\langle\langle L_n\rangle\rangle_{\mathbb T_n}\,.\qquad \qed\]

\begin{remark}
Let $X$ be a normal  variety  with only quotient  singularities, and let $\rho \colon \widehat X\to X$ be a resolution
of  singularities. Then $\rho_*\colon \pi_1(\widehat X) \to \pi_1(X)$ is an isomorphism, 
by \cite[Theorem 7.8]{Kollar}.

In particular, $\pi_1(\widehat{X_{n,d}}) \simeq \pi_1(X_{n,d})$.
\end{remark}

According to the description of $X_{n,d}$  given in the previous section 
its associated  orbifold surface groups $\mathbb T_i$ are all of type  $ [0;d,d,d]$, and applying this discussion to our situation we get the following.

\begin{theorem}
 The universal cover $U_{n,d}$ of $\widehat{ X_{n,d}}$ is projective if and only if $d=4$.
\end{theorem}
\begin{proof}
The universal cover $U_{n,d}$ of $\widehat{ X_{n,d}}$ is projective if and only if the fundamental group 
$\pi_1(\widehat{ X_{n,d}}) $ is finite.
Therefore, by Propositon \ref{prop:fin_crit} $U_{n,d}$ is projective if and only if the groups $\mathbb T_i/\langle\langle L_i\rangle\rangle_{\mathbb T_i}$ are finite.
Let $k:=\frac d2$. 
Since the elements in $\mathbb Z_d^2$ fixing points on $C^n$ are exactly the elements in $H=\langle (k,0), (0,k)\rangle $, by Remark \ref{rem:OSG} ii) we see that $L_i=\{c_1^k,c_2^k, c_3^k\}$, whence
\[\mathbb T_i/\langle\langle L_i\rangle\rangle_{\mathbb T_i} \cong\mathbb T(0;d,d,d)/\langle\langle c_1^k,c_2^k,c_3^k\rangle\rangle= \langle c_1, c_2,c_3 |
 c_1^k,  c_2^k,  c_3^k, c_1 c_2 c_3\rangle\, \cong \mathbb T(0;k,k,k)\,.\]
The statement follows since the group $\mathbb T(0;2,2,2)\cong \mathbb Z_2^2$  is finite, whereas 
$\mathbb T(0;k,k,k)$ is infinite for $k>2$.
\end{proof}

\begin{remark}

i) The first Betti number $b_1$  of $Y_{n,4}$ is zero, because the quotient 
$C/H$ is  isomorphic to the projective line.
Indeed by K\"unneth formula  and \cite[\S 1.2]{McDonald}
we have \[H^1(Y_{n,4}, \mathbb C)=H^1(C^n, \mathbb C)^H
=\bigoplus H^1(C, \mathbb C)^H=\bigoplus H^1(\mathbb P^1, \mathbb C)= 0 \,.\]

Assuming $d=4$, we can actually prove that $g^2=1$ for all $g \in \pi_1(Y_{n,4})=\mathbb G / \Fix(\mathbb G)$. This tells us that 
$\pi_1(Y_{n,4})= \pi_1(\widehat {Y_{n,4}}) \cong \mathbb Z_2^s$ for some $s\in \mathbb N$.

The element $g$ is represented by an $n$-tuple
\[
(w_1, \ldots, w_n) \in \mathbb G= \mathbb T_1\times_H \cdots \times_H \mathbb T_n
\]
where $\mathbb T_k=\mathbb T(0;2,2,2,2,2,2)$ and all the maps
$\varphi_k:\mathbb T_k \to H$ are equal,  as we consider the same action on each factor (see Remark \ref{rem:defH}).
Since $\varphi_k(w_k^2)=(0,0)\in H=\mathbb Z_2^2$, the tuple 
\[
(1, \ldots, 1, w_k^2, 1 \ldots, 1) 
\]
belongs to $\mathbb G$, and to prove the claim it suffices to show that this tuple  is contained in $\Fix(\mathbb G)$. 

Note that the number of occurrences $n_i$ of the letter $c_i$ in the word $w_k^2$ is even.
Observe now,  that in any group a product $a\cdot b$ can be written as $b\cdot (b^{-1} \cdot a \cdot b)$,  hence we can write $w_k^2$ as
\begin{equation}\label{eq:w2}
w_k^2= \bigg(\prod_{i=1}^{n_1} g_{i}^{-1}  c_1 g_{i}\bigg) \cdot \ldots \cdot  \bigg(\prod_{j=1}^{n_6} h_{j}^{-1}  c_6 h_{j} \bigg) \,,
\end{equation}
for certain  $g_i, \ldots,h_j \in \mathbb T_k$. 

By Remark \ref{rem:OSG} ii)
 and since $H$ is abelian,   we get
$( c_1, \ldots, c_1, g_{i}^{-1}  c_1 g_{i}, c_1, \ldots, c_1) 
 \in \Fix(\mathbb G)$.
We conclude that
\[
(1, \ldots, 1, \prod_{i=1}^{n_1} g_{i}^{-1}  c_1 g_{i}, 1 \ldots, 1) 
= \prod_{i=1}^{n_1}  ( c_1, \ldots, c_1, g_{i}^{-1}  c_1 g_{i}, c_1, \ldots, c_1) 
 \in \Fix(\mathbb G)\,.\]
The same applies to each factor in the  RHS of \eqref{eq:w2} and so $(1, \ldots, 1, w_k^2, 1 \ldots, 1)\in \Fix(\mathbb G)$. This shows $g^2=1$, whence $\pi_1(Y_{n,4})$ is abelian, and it is finite since 
 $\pi_1(Y_{n,4})= \pi_1(Y_{n,4})^{ab}=H_1(Y_{n,4}, \mathbb Z)$ has rank $0$.

ii) We implemented Proposition \ref{prop:fund} using the computer algebra system MAGMA \cite{Magma}, and 
we found  $\pi_1(Y_{n,4})= \mathbb Z_2^{n-1}$ and $\pi_1(X_{n,4})= \mathbb Z_2^{n+1}$
for $n=2,3,4,5$. 
 In particular, the  universal cover  of the varieties $X_{n,4}$ and $Y_{n,4}$ has 
$3 \cdot 2^{3n-2}$ singularities of type $\frac{1}{2}(1,\ldots,1)$. 
We expect the above to generalize to  any dimension.
\end{remark}

\begin{remark} \label{rem:final} The surfaces $\widehat{X_{2,d}}$ with $d\geq 8$ are rigid but not infinitesimally rigid (see \cite{BP18}), and their universal cover is non-contractible. This answer partially the question posed in the Introduction in the case of surfaces.
\end{remark}

\end{document}